\documentclass[reqno]{amsart}

\usepackage{amsmath,amssymb,amsthm,mathrsfs,amsfonts,dsfont,stmaryrd}
\usepackage[backref]{hyperref}
\usepackage{multirow}
\usepackage{graphicx}
\usepackage{lscape}

\newtheorem{lemma}{Lemma}[section]
\newtheorem{theorem}[lemma]{Theorem}
\newtheorem{proposition}[lemma]{Proposition}

\newtheorem{corollary}[lemma]{Corollary}

\newtheorem{claim*}{Claim}
\newtheorem{remark}[lemma]{Remark}

\newtheorem{definition}[lemma]{Definition}

\def\O{\mathcal{O}}
\newcommand{\Q}{{\mathbb Q}}
\newcommand{\C}{{\mathbb C}}
\newcommand{\Z}{{\mathbb Z}}
\newcommand{\F}{{\mathbb F}}
\newcommand{\Inv}{\operatorname{Inv}}
\newcommand{\into}{\hookrightarrow}
\DeclareMathOperator{\Sp}{Sp}
\DeclareMathOperator{\Jac}{Jac}
\DeclareMathOperator{\GL}{GL}
\DeclareMathOperator{\Sym}{Sym}
\DeclareMathOperator{\SL}{SL}
\DeclareMathOperator{\Disc}{Disc}
\DeclareMathOperator{\diag}{diag}
\DeclareMathOperator{\End}{End}

\newcommand{\Kilicer}{K\i{}l\i{}\c{c}er}
\newcommand{\Manzateanu}{M\^{a}nz\u{a}\c{t}eanu}

\numberwithin{equation}{section}
\numberwithin{table}{section}

\begin{document}

\title{Modular invariants for genus 3 hyperelliptic curves
}

\author{Sorina Ionica}
\address{Sorina Ionica, Laboratoire MIS, Universit\'e de Picardie Jules Verne, 33 Rue Saint Leu, Amiens 80039, France}
\email{sorina.ionica@u-picardie.fr}

\author{P{\i}nar K{\i}l{\i}\c{c}er} 
\thanks{Most of K{\i}l{\i}\c{c}er's work was carried out during her stay in Universiteit Leiden and Carl von Ossietzky Universit\"at Oldenburg. }
\address{P{\i}nar K{\i}l{\i}\c{c}er,
Johann Bernoulli Instituut voor Wiskunde en Informatica, 
Rijksuniversiteit Groningen, Nijenborgh 9,
9747 AG Groningen, Netherlands}
\email{p.kilicer@rug.nl}

\author{Kristin Lauter}
\address{Kristin Lauter, Microsoft Research, One Redmond Way, Redmond, WA 98052, USA}
\email{klauter@microsoft.com}

\author{Elisa Lorenzo Garc\'ia}
\address{Elisa Lorenzo Garc\'ia, Laboratoire IRMAR, Office 602,
Universit\'e de Rennes 1, Campus de Beaulieu, 35042, Rennes Cedex, France}
\email{elisa.lorenzogarcia@univ-rennes1.fr }

\author{Maike Massierer}
\thanks{Maike Massierer was supported by the Australian Research Council (DP150101689).}
\address{Maike Massierer, School of Mathematics and Statistics, University of New South Wales, Sydney NSW 2052, 
Australia}
\email{maike.massierer@gmail.com}

\author{Adelina \Manzateanu}
\address{Adelina \Manzateanu, Institute of Science and Technology Austria. School of Mathematics, University of Bristol, Bristol, BS8 1JR, UK}
\email{am15112@bristol.ac.uk}

\author{Christelle Vincent}
\thanks{Vincent is supported by the National Science Foundation under Grant No. DMS-1802323.}
\address{Christelle Vincent, Department of Mathematics and Statistics, University of Vermont, 16 Colchester Avenue, Burlington VT 05401}
\email{christelle.vincent@uvm.edu}

\makeatletter
\let\@wraptoccontribs\wraptoccontribs
\makeatother


\maketitle
\markleft{S. IONICA ET AL.}
\begin{abstract}
In this article we prove an analogue of a theorem of Lachaud, Ritzenthaler, and Zykin, which allows us to connect invariants of binary octics to Siegel modular forms of genus 3. We use this connection to show that certain modular functions, when restricted to the hyperelliptic locus, assume values whose denominators are products of powers of primes of bad reduction for the associated hyperelliptic curves. We illustrate our theorem with explicit computations. This work is motivated by the study of the values of these modular functions at CM points of the Siegel upper-half space, which, if their denominators are known, can be used to effectively compute models of (hyperelliptic, in our case) curves with CM.
\end{abstract}

\section{Introduction} \label{sec:intro}

In his beautiful paper, Igusa~\cite{igusa67} proved that there is a homomorphism from a subring (containing forms of even weight) of the graded ring of Siegel modular forms of genus $g$ and level $1$ to the graded ring of invariants of binary forms of degree $2g+2$.
In this paper, we consider Siegel modular functions which map to invariants of hyperelliptic curves under this homomorphism, and are thus called {\it modular invariants}.

We are interested in the primes that divide the denominators of certain quotients of these modular invariants.\footnote{Here by denominator we mean the least common multiple of the (rational) denominators that appear in an algebraic number's monic minimal polynomial.} Our work is motivated by the following computational problem: To recognize the value of a modular invariant as an exact algebraic number from a floating point approximation, one must have a bound on its denominator. Furthermore, the running time of the algorithm is greatly improved when the bound is tight.

Igusa~\cite{igusa67} gave an explicit construction of the above-mentioned homomorphism for all modular forms of level 1 which can be written as polynomials in the theta-constants. Our first contribution is an analogue of a result of Lachaud, Ritzenthaler, and Zykin \cite[Corollary~3.3.2]{LRZ}, which connects Siegel modular forms to invariants of plane quartics. Using a similar approach, which first connects Siegel modular forms to Teichm\"uller modular forms, we obtain a construction which is equivalent to Igusa's for modular forms of even weight. We then compute the image of the discriminant of a hyperelliptic curve under this homomorphism, thus extending and rephrasing a result of Lockhart~\cite[Proposition 3.2]{Lockhart}. This allows us to prove our main theorem:

\begin{theorem}\label{prop:newinvariant}
Let $Z$ be a period matrix in $\mathcal{H}_3$, the Siegel upper half-plane of genus $3$, corresponding to a smooth genus 3 hyperelliptic curve $C$ defined over a number field $M$. Let $f$ be a Siegel modular form of weight $k$ such that the invariant $\Phi$ obtained in Corollary~\ref{cor:3.3.2} is integral. Then
\begin{equation*}
j(Z) = \frac{f^{\frac{140}{\gcd(k,140)}}}{\Sigma_{140}^{\frac{k}{\gcd(k,140)}}}(Z)
\end{equation*}
is an algebraic number lying in $M$. 
Moreover, if an odd prime $\mathfrak{p}$ of $\mathcal{O}_{M}$ divides the denominator of this number, then the curve $C$ has geometrically bad reduction modulo $\mathfrak{p}$.
\end{theorem}

Here, $\Sigma_{140}$ is the Siegel modular form of genus $3$ defined by Igusa \cite{igusa67} in terms of the theta constants (see equation~(\ref{eq:Thetas})) as follows:
\begin{equation}\label{eq:sigma140}
\Sigma_{140}(Z) = \sum_{i = 1}^{36} \prod_{j \neq i} \vartheta[\xi_j](0,Z)^8,
\end{equation}
where the $\xi_i$, $i=1,\ldots,36$ are the even theta characteristics we define in Section~\ref{Sec:CM}.

To illustrate this theorem, in Section~\ref{Sec-NewInv} we compute values of several modular invariants whose expressions have a power of $\Sigma_{140}$ in the denominator. For our experiments, we used: genus 3 hyperelliptic CM curves defined over $\Q$, a complete list of which is given in \cite{KS2017}; genus 3 hyperelliptic curves already appearing in some experiments concerning the Chabauty-Coleman method \cite{Chabauty}; and some genus~3 hyperelliptic modular curves \cite{GalbraithThesis,Ogg}.

Note that Theorem~\ref{prop:newinvariant} is an analogue of a result of Goren and Lauter for curves of genus 2 with CM~\cite{GorenLauter07}.
The case of CM hyperelliptic curves is interesting because the bound on the primes dividing the denominators of Igusa invariants proved in~\cite{GorenLauter07} is used to improve the algorithms to construct genus $2$ CM curves.  
We hope that apart from its theoretical interest, our result will allow a similar computation in the case of CM hyperelliptic curves of genus $3$.

\paragraph{Outline.}
This paper is organized as follows. We begin in Section \ref{Sec:CM} with some background on theta functions, the Igusa construction and the Shioda invariants of hyperelliptic curves. Only the most basic facts are given, and references are provided for the reader who would like to delve further.

Then, in Section \ref{Sec:ModularInvs}, we give a correspondence that allows us to relate invariants of octics to Siegel modular forms of degree 3. Using this correspondence, we then show in Section \ref{Sec:MainTheorem} that the primes dividing the denominators of modular invariants that have powers of the Siegel modular form $\Sigma_{140}$ as their denominator are primes of bad reduction, which is our main theorem (Theorem \ref{prop:newinvariant} above). 

Finally, in Section~\ref{Sec-NewInv} we present the list of hyperelliptic curves of genus $3$ for which we computed the values of several modular invariants having powers of $\Sigma_{140}$ as their denominator, when evaluated at a period matrix of their Jacobian. We compared the factorization of the denominators of these values against that of the denominators of the Shioda invariants of these curves and the odd primes of bad reduction of these curves. 

\section*{acknowledgements}
The authors would like to thank the Lorentz Center in Leiden for hosting the Women in Numbers Europe 2 workshop and providing a productive and enjoyable environment for our initial work on this project. We are grateful to the organizers of WIN-E2, Irene Bouw, Rachel Newton and Ekin Ozman, for making this conference and this collaboration possible. We thank Irene Bouw and Christophe Ritzenhaler for helpful discussions.

\section{Hyperelliptic curves of genus $3$ with complex multiplication }\label{Sec:CM}

In this section we introduce notation and  discuss theta functions and theta characteristics, which are crucial to the definition of the Siegel modular invariants we consider in this paper. We briefly recall Igusa's contruction of a homomorphism between the graded ring of Siegel modular forms and the graded ring of invariants of a binary form. Finally, we define the Shioda invariants of genus 3 hyperelliptic curves.

\subsection{Theta functions and theta characteristics} 
\label{Sec: hyp}

In this work, by \emph{period matrix} we will mean a $g \times g$ symmetric matrix $Z$ with positive imaginary part, that is, a matrix in in the Siegel upper half-space of genus $g$. (This is sometimes called a \emph{small} period matrix, but for simplicity and since there is no risk of confusion here we call them period matrices.)

In this case, the relationship between the abelian variety and the period matrix is that the complex points of the abelian variety are exactly the complex points of the torus $\C^g/(\mathbb{Z}^g+Z\mathbb{Z}^g)$. 

We denote by $\mathcal{H}_g$ the Siegel upper half space. 
We now turn our attention to the subject of theta functions. For $\omega \in \C^g$ and $Z \in \mathcal{H}_g$, we define the following important series:
\begin{equation*}
\vartheta(\omega, Z) = \sum_{n \in \Z^{g}}\exp(\pi i n^T Z n + 2 \pi i n^ T \omega),
\end{equation*}
where throughout this article an exponent of $T$ on a vector or a matrix denotes the transpose.

Given a period matrix $Z \in \mathcal{H}_g$, we obtain a set of coordinates on the torus $\C^g/(\mathbb{Z}^g+Z\mathbb{Z}^g)$ in the following way: A vector $x \in [0,1)^{2g}$ corresponds to the point $x_2 +  Z x_1 \in \C^g/(\mathbb{Z}^g+Z\mathbb{Z}^g)$, where $x_1$ denotes the first $g$ entries and $x_2$ denotes the last $g$ entries of the vector $x$ of length $2g$.

For reasons beyond the scope of this short text, it is of interest to consider the value of this theta function as we translate $\omega$ by points that, under the natural quotient map $\C^g \to \C^g/(\mathbb{Z}^g+Z\mathbb{Z}^g)$, map to $2$-torsion points. These points are of the form $\xi_2 + Z \xi_1$ for $\xi \in (1/2)\Z^{2g}$. This motivates the following definition:
\begin{equation}\label{eq:Thetas}
\vartheta[\xi](\omega, Z) = \exp(\pi i \xi_1^T Z \xi_1 + 2 \pi i \xi_1^T (\omega+\xi_2)) \vartheta( \omega+\xi_2 + Z \xi_1, Z),
\end{equation}
which is given in~\cite[page 123]{Mumford1}.
In this context, $\xi$ is customarily called a \emph{characteristic} or \emph{theta characteristic}. The value $\vartheta[\xi](0,Z)$ is called a \emph{theta constant}.

For $\xi \in (1/2)\Z^{2g}$, let 
\begin{equation}\label{eq:estar}
e_*(\xi) = \exp(4\pi i \xi_1^T  \xi_2).
\end{equation}
We say that a characteristic $\xi \in (1/2)\Z^{2g}$ is \emph{even} if $e_*(\xi) = 1$ and \emph{odd} if $e_*(\xi) = -1$. If $\xi$ is even we call $\vartheta[\xi](0,Z)$ an \emph{even theta constant} and if $\xi$ is odd we call $\vartheta[\xi](0,Z)$ an \emph{odd theta constant}.

We have the following fact about the series $\vartheta[\xi](\omega,Z)$ \cite[Chapter II, Proposition 3.14]{Mumford1}: For $\xi \in (1/2)\Z^{2g}$,
\begin{equation*}
\vartheta[\xi](-\omega,Z) = e_*(\xi) \vartheta[\xi](\omega,Z).
\end{equation*}
From this we conclude that all odd theta constants vanish. Furthermore, we have that if $n \in \mathbb{Z}^{2g}$ is a vector with integer entries,
\begin{equation*}
\vartheta[\xi + n](\omega, Z) = \exp(2\pi i \xi_1^T n_2) \vartheta[\xi](\omega, Z).
\end{equation*}
In other words, if $\xi$ is modified by a vector with integer entries, the theta value at worst acquires a factor of~$-1$. Up to this sign, we note that there are in total $2^{g-1}(2^g+1)$ even theta constants and $2^{g-1}(2^g-1)$ odd ones.

We can now finally fully describe the modular form $\Sigma_{140}$ defined in the introduction (equation \eqref{eq:sigma140}). First, we note that when $g = 3$, there are $36$ even theta characteristics. For simplicity of notation, we give an arbitrary ordering to these even theta characteristics, and label them $\xi_1, \ldots, \xi_{36}$. Then we have
\begin{equation*}
\Sigma_{140}(Z) = \sum_{i = 1}^{36} \prod_{j \neq i} \vartheta[\xi_j](0,Z)^8,
\end{equation*}
the $35^{\text{th}}$ elementary symmetric polynomial in the even theta constants.

We will also need another Siegel modular form introduced by Igusa \cite{igusa67} and given by 
\begin{equation}\label{eq:chi18}
\chi_{18}(Z) = \prod_{i=1}^{36} \vartheta[\xi_i](0,Z).
\end{equation} 
Igusa shows that $\Sigma_{140}$ and $\chi_{18}$ are Siegel modular forms for the symplectic group of level 1 $\Sp(6,\mathbb{Z})$.

The significance of these modular forms is the following: in \emph{loc. cit}, Igusa shows that a period matrix $Z$ corresponds to a simple Jacobian of hyperelliptic curve when $\chi_{18}(Z)=0$ and $\Sigma_{140}(Z)\neq 0$ and it is a reducible Jacobian when $\chi_{18}(Z)=\Sigma_{140}(Z)=0$. Moreover, $\chi_{18}$ will appear later as the kernel of Siegel's homomorphism mentioned in the introduction.

\subsection{Igusa's construction}

Let $S(2,2g+2)$ be the graded ring of projective invariants of a binary form of degree $2g+2$. We denote by $\Sp(2g,\mathbb{Z})$ the symplectic group of matrices of dimension $2g$ and by $A(\Sp(2g,\mathbb{Z}))$ the graded ring of modular forms of degree $g$ and level 1. There exists a homomorphism 
\begin{eqnarray*}
\rho \colon A(\Sp(2g,\mathbb{Z}))\rightarrow S(2,2g+2),
\end{eqnarray*}
which was first constructed by Igusa~\cite{igusa67}. Historically, Igusa only showed that the domain of $\rho$ equals $A(\Sp(2g,\mathbb{Z}))$ when $g$ is odd or $g=2,4$, and that for even $g>4$, a sufficient condition for the domain to be the full ring $A(\Sp(2g,\mathbb{Z}))$ is the existence of a modular form of odd weight that does not vanish on the hyperelliptic locus. Such a form was later exhibited by Salvati Manni in~\cite{Salvati}, from which it follows that the domain of $\rho$ is the full ring of Siegel modular forms.

The kernel of $\rho$ is given by modular forms which vanish on all points in $\mathcal{H}_g$ associated with a hyperelliptic curve. In particular, Igusa shows that in genus 3, the kernel of $\rho$ is a principal ideal generated by the form $\chi_{18}$ defined in equation \eqref{eq:chi18}. Furthermore, Igusa shows that this homomorphism $\rho$ is unique, up to a constant. More precisely, any other map is of the form $\zeta_4^{k}\rho$ on the homogenous part $A(\Sp(2g,\mathbb{Z}))_k$, where $\zeta_4$ is a fourth root of unity.
In Section~\ref{Sec:ModularInvs} we display a similar map sending Siegel modular forms to invariants, by going first through the space of geometric Siegel modular forms and then through that of Teichm\"uller forms. As a consequence, our map coincides with the map $\rho$ constructed by Igusa, up to constants. The advantage of our construction is that it allows us to identify a modular form that is in the preimage of a power of the discriminant of the curve under this homomorphism.

\subsection{Shioda invariants}\label{Sec:Shiodas}

We lastly turn our attention to the (integral) invariants under study in this article. We say that polynomials in the coefficients of a binary form corresponding to a hyperelliptic curve that are invariant under the natural action of $\operatorname{SL}_2(\mathbb{C})$ are \emph{invariants of the hyperelliptic curve}, and furthermore that such an invariant is \emph{integral} if the polynomial has integer coefficients. Shioda gave a set of generators for the algebra of invariants of binary octics over the complex numbers~\cite{Shioda}, which are now called \emph{Shioda invariants}. In addition, over the complex numbers, Shioda invariants completely classify isomorphism classes of hyperelliptic curves of genus $3$.
More specifically, the Shioda invariants are $9$ weighted projective invariants $(J_2,J_3,J_4,J_5,J_6,J_7,J_8,J_9,J_{10})$, where $J_i$ has degree $i$, and $J_2, \ldots, J_7$ are algebraically independent, while $J_8,J_9,J_{10}$ depend algebraically on the previous Shioda invariants.

In \cite{LerRit}, the authors showed that these invariants are also generators of the algebra of invariants and determine hyperelliptic curves of genus 3 up to isomorphism in characteristic $p>7$. Later, in his thesis \cite{basson}, Basson provided some extra invariants that together with the classical Shioda invariants classify hyperelliptic curves of genus~3 up to isomorphism in characteristics $3$ and $7$. The characteristic $5$ case is still an unpublished theorem of Basson. 

\section{Invariants of hyperelliptic curves and Siegel modular forms}\label{Sec:ModularInvs}

The aim of this section is to establish an analogue for the hyperelliptic locus of Corollary~3.3.2 in an article of Lachaud, Ritzenthaler and Zykin \cite{LRZ}. Our result, while technically new, does not use any ideas that do not appear in the original paper. We begin by establishing the basic ingredients necessary, using the same notation as in \cite{LRZ} for clarity, and with the understanding that, when omitted, all details may be found in \emph{loc.\ cit.}

Roughly speaking, the main idea of the proof is to compare three different ``flavors'' of modular forms and invariants of non-hyperelliptic curves (which will here be replaced with invariants of hyperelliptic curves). The comparison goes as follows: to connect analytic Siegel modular forms to invariants of curves, the authors first connect analytic Siegel modular forms to geometric modular forms. Following this, geometric modular forms are connected to Teichm\"uller modular forms, via the Torelli map and a result of Ichikawa. Finally Teichm\"uller forms are connected to invariants of curves.

\subsection{From analytic Siegel modular forms to geometric Siegel modular forms}
Let $\mathbf{A}_{g}$ be the moduli stack of principally polarized abelian schemes of relative dimension $g$, and $\pi \colon \mathbf{V}_{g} \to \mathbf{A}_{g}$ be the universal abelian scheme with zero section $\epsilon \colon \mathbf{A}_{g} \to \mathbf{V}_{g}$. Then the relative canonical line bundle over $\mathbf{A}_g$ is given in terms of the rank $g$ bundle of relative regular differential forms of degree one on $\mathbf{V}_g$ over $\mathbf{A}_g$ by the expression
\begin{equation*}
\boldsymbol{\omega} = \bigwedge^g \epsilon^*\Omega^1_{\mathbf{V}_g/\mathbf{A}_g}.
\end{equation*}

With this notation, a \emph{geometric Siegel modular form of genus $g$ and weight $h$}, for $h$ a positive integer, over a field $k$, is an element of the $k$-vector space
\begin{equation*}
\mathbf{S}_{g,h}(k) = \Gamma(\mathbf{A}_g \otimes k, \boldsymbol{\omega}^{\otimes h}).
\end{equation*}
If $f \in \mathbf{S}_{g,h}(k)$ and $A$ is a principally polarized abelian variety of dimension $g$ defined over $k$ equipped with a basis $\alpha$ of the 1-dimensional space $\boldsymbol{\omega}_{k}(A)=\bigwedge^g \Omega^1_k(A)$, we define
\begin{equation*}
f(A, \alpha) = \frac{f(A)}{\alpha^{\otimes h}}.
\end{equation*}
In this way $f(A, \alpha)$ is an algebraic or geometric modular form in the usual sense, i.e.,
\begin{enumerate}
\item $f(A,\lambda \alpha) = \lambda^{-h} f(A,\alpha)$ for any $\lambda \in k^{\times}$, and
\item $f(A,\alpha)$ depends only on the $\bar{k}$-isomorphism class of the pair $(A,\alpha)$.
\end{enumerate}
Conversely, such a rule defines a unique $f \in \mathbf{S}_{g,h}$.

We first compare these geometric Siegel modular forms to the usual analytic Siegel modular forms:

\begin{proposition}[Proposition 2.2.1 of \cite{LRZ}]\label{prop:2.2.1}
Let $\mathbf{R}_{g,h}(\mathbb{C})$ denote the usual space of analytic Siegel modular forms of genus $g$ and weight $h$. Then there is an isomorphism
\begin{equation*}
\mathbf{S}_{g,h}(\mathbb{C}) \to \mathbf{R}_{g,h}(\mathbb{C}),
\end{equation*}
given by sending $f \in \mathbf{S}_{g,h}(\mathbb{C})$ to
\begin{equation*}
\tilde{f}(Z) = \frac{f(A_{Z})}{(2\pi i)^{gh}(dz_1 \wedge \ldots \wedge dz_g)^{\otimes h}},
\end{equation*}
where $A_{Z} = \mathbb{C}^g/(\mathbb{Z}^g + Z \mathbb{Z}^g)$, $Z \in \mathcal{H}_g$ and each $z_i \in \mathbb{C}$.
\end{proposition}

Furthermore, this isomorphism has the following pleasant property:

\begin{proposition}[Proposition 2.4.4 of \cite{LRZ}]\label{prop:2.4.4}
Let $(A,a)$ be a principally polarized abelian variety of dimension $g$ defined over $\mathbb{C}$, let $\omega_1, \ldots, \omega_g$ be a basis of $\Omega^1_{\mathbb{C}}(A)$ and let $\omega = \omega_1 \wedge \ldots \wedge \omega_g \in \boldsymbol{\omega}_{\mathbb{C}}(A)$. If $\Omega = \left(\begin{smallmatrix} \Omega_1 & \Omega_2 \end{smallmatrix} \right)$ is a Riemann matrix obtained by integrating the forms $\omega_i$ against a basis of $H_1(A,\mathbb{Z})$ for the polarization $a$, then $Z = \Omega_2^{-1}\Omega_1$ is in $\mathcal{H}_g$ and
\begin{equation*}
f(A, \omega) = (2 \pi i)^{gh} \frac{\tilde{f}(Z)}{\det \Omega_2^h}.
\end{equation*}
\end{proposition}

\subsection{From geometric Siegel modular forms to Teichm\"uller modular forms}
We now turn our attention to so-called Teichm\"uller modular forms, which were studied by Ichikawa \cite{ichikawa1}\cite{ichikawa2}\cite{ichikawa3}\cite{ichikawa4}. Let $\mathbf{M}_g$ be the moduli stack of curves of genus $g$, let $\pi \colon \mathbf{C}_g \to M_g$ be the universal curve, and let
\begin{equation*}
\boldsymbol{\lambda} = \bigwedge^g \pi_* \Omega^1_{\mathbf{C}_g/\mathbf{M}_g}
\end{equation*}
be the invertible sheaf associated to the Hodge bundle. 

With this notation, a \emph{Teichm\"uller modular form of genus $g$ and weight $h$}, for $h$ a positive integer, over a field $k$, is an element of the $k$-vector space
\begin{equation*}
\mathbf{T}_{g,h}(k) = \Gamma(\mathbf{M}_g \otimes k,\boldsymbol{\lambda}^{\otimes h}).
\end{equation*}
As before, if $f \in \mathbf{T}_{g,h}(k)$ and $C$ is a curve of genus $g$ defined over $k$ equipped with a basis $\lambda$ of $\boldsymbol{\lambda}_{k}(C)=\bigwedge^g \Omega^1_k(C)$, we define
\begin{equation*}
f(C, \lambda) = \frac{f(C)}{\lambda^{\otimes h}}.
\end{equation*}
Again, $f(C, \lambda)$ is an algebraic modular form in the usual sense.
Ichikawa proves:
\begin{proposition}[Proposition 2.3.1 of \cite{LRZ}]\label{prop:2.3.1}
The Torelli map $\theta \colon \mathbf{M}_g \to \mathbf{A}_g$, associating to a curve $C$ its Jacobian $\Jac C$ with the canonical polarization $j$, satisfies $\theta^* \boldsymbol{\omega} = \boldsymbol{\lambda}$, and induces for any field a linear map
\begin{equation*}
\theta^* \colon \mathbf{S}_{g,h}(k) \to \mathbf{T}_{g,h}(k)
\end{equation*}
such that $(\theta^*f)(C) = \theta^*(f(\Jac C)).$ In other words, for a basis $\lambda$ of $\boldsymbol{\lambda}_{k}(C)$ and fixing $\alpha$ such that 
 a basis $\alpha$ of $\boldsymbol{\omega}_k (C)$ whose pullback to $C$ equals $\lambda$,
\begin{equation*}
f(\Jac C, \alpha) = (\theta^* f)(C,\lambda).
\end{equation*}
\end{proposition}

\subsection{From Teichm\"uller modular forms to invariants of binary forms}
We finally connect the Teichm\"uller modular forms to invariants of hyperelliptic curves. To this end, let $E$ be a vector space of dimension $2$ over a field $k$ of characteristic different from $2$, and put $G = \GL(E)$ and $\mathbf{X}_d = \Sym^d(E^*)$, the space of homogeneous polynomials of degree $d$ on $E$. We define the action of $G$ on $\mathbf{X}_d$, $u \cdot F$ for $F \in \mathbf{X}_d$, by
\begin{equation*}
(u \cdot F)(x,z) = F(u^{-1}(x,z)).
\end{equation*}
(By a slight abuse of notation we denote an element of $E$ by the pair $(x,z)$, effectively prescribing a basis. Our reason to do so will become clear later.)

We say that $\Phi$ is an \emph{invariant of degree $h$} if $\Phi$ is a regular function on $\mathbf{X}_d$, homogeneous of degree $h$ (by which we mean that $\Phi(\lambda F) = \lambda^h \Phi(F)$ for $\lambda \in k^{\times}$ and $F \in \mathbf{X}_d$) and
\begin{equation*}
u \cdot \Phi = \Phi \quad \text{for every} \quad u \in \SL(E),
\end{equation*}
where the action $u \cdot \Phi$ is given by 
\begin{equation*}
(u \cdot \Phi)(F) = \Phi( u^{-1} \cdot F).
\end{equation*}
We note the space of invariants of degree $h$ by $\Inv_h(\mathbf{X}_d)$. Note that in what follows we will define an open set of $\mathbf{X}_d^0$, and be interested in the invariants of degree $h$ that are regular on that open set. The definition of invariance is the same, all that changes is the set on which the function is required to be regular.

From now on we require $d\geq 6$ to be even, and put $g = \frac{d-2}{2}$, then the \emph{universal hyperelliptic curve} over the the affine space $\mathbf{X}_d = \Sym^d(E)$ is the variety
\begin{equation*}
\mathbf{Y}_d = \left\{(F,(x,y,z)) \in \mathbf{X}_d \times \mathbb{P}\left(1,\frac{d}{2},1\right) : y^2 = F(x,z) \right\},
\end{equation*}
where $\mathbb{P}(1,g+1,1)$ is the weighted projective plane with $x$ and $z$ having weight $1$ and $y$ having weight $g+1$. The non-singular locus of $\mathbf{X}_d$ is the open set
\begin{equation*}
\mathbf{X}_d^0 = \{ F \in \mathbf{X}_d : \Disc(F) \neq 0\}.
\end{equation*}
We denote by $\mathbf{Y}_d^0$ the restriction of $\mathbf{Y}_d$ to the nonsingular locus. The projection gives a smooth surjective $k$-morphism 
\begin{equation*}
\pi \colon \mathbf{Y}_d^0 \to \mathbf{X}_d^0
\end{equation*}
and its fiber over $F$ is the nonsingular hyperelliptic curve $C_F : y^2 = F(x,z)$ of genus $g$. In this case we have en explicit $k$-basis for the space of holomorphic differentials of $C_F$, denoted $\Omega^1(C_F)$, given by
\begin{equation}\label{eq:basiseta}
\omega_1 = \frac{dx}{y}, \omega_2 = \frac{xdx}{y}, \ldots, \omega_g = \frac{x^{g-1}dx}{y}.
\end{equation}

Now let $u \in G$ act on $\mathbf{Y}_d$ by
\begin{equation*}
u \cdot (F, (x,y,z)) = (u \cdot F, u \cdot (x,y,z)),
\end{equation*}
where the action on $F$ is given by
\begin{equation*}
(u \cdot F) (x,z) = F(u^{-1}(x,z))
\end{equation*}
and the action of $u$ on $(x,y,z)$ is given by replacing the vector $(x,z)$ by $u(x,z)$ and leaving $y$ invariant. Then the projection 
\begin{equation*}
\pi \colon \mathbf{Y}_d^0 \to \mathbf{X}_d^0
\end{equation*}
is $G$-equivariant.

Then as in \cite{LRZ}, the section
\begin{equation*}
\omega = \omega_1 \wedge \ldots \wedge \omega_g
\end{equation*}
is a basis of the one-dimensional space $\Gamma(\mathbf{X}^0_d, \boldsymbol{\alpha})$, where 
\begin{equation*}
\boldsymbol{\alpha} = \bigwedge^g \pi_* \Omega^1_{\mathbf{Y}_d^0/\mathbf{X}_d^0},
\end{equation*}
the Hodge bundle of the universal curve over $\mathbf{X}_d^0$. For every $F \in \mathbf{X}_d^0$, an element $u \in G$ induces an isomorphism
\begin{equation*}
\phi_u \colon C_F \to C_{u\cdot F},
\end{equation*}
and this defines a linear automorphism $\phi^*_u$ of $\boldsymbol{\alpha}$.

For any $h \in \mathbb{Z}$, we define $\Gamma(\mathbf{X}_d^0, \boldsymbol{\alpha}^{\otimes h})^G$ the subspace of sections $s \in \Gamma(\mathbf{X}_d^0, \boldsymbol{\alpha}^{\otimes h})$ such that 
\begin{equation*}
\phi_u^*(s) = s
\end{equation*}
for every $u \in G$. Then if $\alpha \in \Gamma(\mathbf{X}_d^0, \boldsymbol{\alpha})$ and $F \in \mathbf{X}_d^0$, we define
\begin{equation*}
s(F,\alpha) = \frac{s(F)}{\alpha^{\otimes h}}.
\end{equation*}
This gives us the space that will be related to invariants of hyperelliptic curves, which we now define.

In this setting we have the exact analogue of Proposition 3.2.1 of \cite{LRZ}:
\begin{proposition}\label{prop:3.2.1}
The section $\omega \in \Gamma(X_d^0,\boldsymbol{\alpha})$ satisfies the following properties:
\begin{enumerate}
\item If $u \in G$, then
\begin{equation*}
\phi_u^* \omega = \det(u)^{w_0} \omega,
\end{equation*}
with
\begin{equation*}
w_0 = \frac{dg}{4}.
\end{equation*}
\item Let $h \geq 0$ be an integer. The linear map
\begin{align*}
\tau \colon \Inv_{\frac{gh}{2}}(\mathbf{X}_d^0) &\to \Gamma(\mathbf{X}^0_d, \boldsymbol{\alpha}^{\otimes h})^G \\
\Phi &\mapsto \Phi \cdot \omega^{\otimes h}
\end{align*}
is an isomorphism.
\end{enumerate}
\end{proposition}

\begin{proof}
The proof of the first part goes exactly as in the original: For $u \in G$, we have that
\begin{equation*}
(\phi_u^* \omega) (F, \omega) = c(u,F) \omega(F,\omega),
\end{equation*}
and we can conclude, via the argument given in \cite{LRZ}, that $c(u,F)$ is independent of $F$ and a character $\chi$ of $G$, and that in fact
\begin{equation*}
c(u,F) = \chi(u) = \det u ^{w_0}
\end{equation*}
for some integer $w_0$. To compute $w_0$ we again follow the original and set $u = \lambda I_2$ with $\lambda \in k^{\times}$ to obtain
\begin{equation*}
\frac{\omega_i(\lambda^{-d}F)}{\omega_i(F)} = \frac{x^{i-1}dx}{\sqrt{\lambda^{-d}F(x,y)}}  \div \frac{x^{i-1}dx}{\sqrt{F(x,y)}} = \lambda^{d/2},
\end{equation*}
since $y = \sqrt{F(x,y)}$, for each $i = 1, \ldots, g$. Hence
\begin{equation*}
(\phi_u^* \omega) (F, \omega) = \lambda^{dg/2} = \det(u)^{w_0}
\end{equation*}
and since $\det(u) = \lambda^2$ we have
\begin{equation*}
{w_0} = \frac{dg}{4} = \frac{d(d-2)}{8}.
\end{equation*}

The proof of the second part also goes exactly as in the original, with the replacement of a denominator of $4$ instead of $3$ in the quantity that is denoted $w$ in \cite{LRZ}.
\end{proof}

\subsection{Final step}
With this in hand, we immediately obtain the analogue of Proposition 3.3.1 of \cite{LRZ}. We begin by setting up the notation we will need. We continue to have $d \geq 6$ an even integer and $g = \frac{d-2}{2}$. Because the fibers of $\pi \colon \mathbf{Y}_d^0 \to \mathbf{X}_d^0$ are smooth hyperelliptic curves of genus $g$, by the universal property of $\mathbf{M}_g$, we get a morphism
\begin{equation*}
p \colon \mathbf{X}_g^0 \to \mathbf{M}_g^{hyp},
\end{equation*}
where this time $\mathbf{M}_g^{hyp}$ is the hyperelliptic locus of the moduli stack $\mathbf{M}_g$ of curves of genus $g$. By construction we have $p^* \boldsymbol{\lambda} = \boldsymbol{\alpha}$, and therefore we obtain a morphism
\begin{equation*}
p^* \colon \Gamma(\mathbf{M}_g^{hyp}, \boldsymbol{\lambda}^{\otimes h}) \to \Gamma(\mathbf{X}_d^0, \boldsymbol{\alpha}^{\otimes h}).
\end{equation*}
As in \cite{LRZ}, by the universal property of $\mathbf{M}_g^{hyp}$, we have
\begin{equation*}
\phi_u^* \circ p^*(s) = p^*(s)
\end{equation*}
for $s \in \Gamma(\mathbf{M}_g^{hyp}, \boldsymbol{\lambda}^{\otimes h})$. From this we conclude that $p^*(s) \in \Gamma(X_d^0,\boldsymbol{\alpha})^G$, and combining this with the second part of Proposition \ref{prop:3.2.1}, which establishes the isomorphism of $\Gamma(X_d^0,\boldsymbol{\alpha})^G$ and $\Inv_{gh}(X_d^0)$, we obtain:

\begin{proposition}\label{prop:3.3.1}
For any even $h \geq 0$, the linear map given by $\sigma = \tau^{-1} \circ p^*$ is a homomorphism
\begin{equation*}
\sigma \colon \Gamma(\mathbf{M}_g^{hyp}, \boldsymbol{\lambda}^{\otimes h}) \to \Inv_{\frac{gh}{2}}(X_d^0)
\end{equation*}
satisfying
\begin{equation*}
\sigma(f)(F) = f(C_F, (p^*)^{-1}\omega)
\end{equation*}
for any $F \in \mathbf{X}_d^0$ and any section $f \in \Gamma(\mathbf{M}_g^{hyp}, \boldsymbol{\lambda}^{\otimes h})$.
\end{proposition}

This is the last ingredient necessary to show the analogue of Corollary 3.3.2 of \cite{LRZ}.

\begin{corollary}\label{cor:3.3.2}
Let $f \in \mathbf{S}_{g,h}(\mathbb{C})$ be a geometric Siegel modular form, $\tilde{f} \in \mathbf{R}_{g,h}(\mathbb{C})$ be the corresponding analytic modular form, and $\Phi = \sigma(\theta^*f)$ the corresponding invariant. Let further $F \in \mathbf{X}_d^0$ give rise to the curve $C_F$ equipped with the basis of regular differentials given by the forms $\omega_1, \ldots, \omega_g$ given in equation \eqref{eq:basiseta}. Then if $\Omega = \left( \begin{smallmatrix} \Omega_1 & \Omega_2\end{smallmatrix}\right)$ is a Riemann matrix for the curve $C_F$ obtained by integrating the forms $\omega_i$ against a symplectic basis for the homology group $H_1(C_F,\mathbb{Z})$ and $Z = \Omega_2^{-1}\Omega_1 \in \mathbb{H}_g$, we have
\begin{equation*}
\Phi(F) = (2 i \pi)^{gh} \frac{\tilde{f}(Z)}{\det \Omega_2^h}.
\end{equation*}
\end{corollary}

The last two results display a connection between Siegel modular forms of even weight restricted to the hyperelliptic locus and invariants of binary forms of degree $2g+2$.

\section{Denominators of modular invariants and primes of bad reduction}\label{Sec:MainTheorem}

In this Section we prove our main theorem, Theorem \ref{prop:newinvariant}. The proof of this result has three main ingredients. In the previous Section, we have already adapted to the case of hyperelliptic curves a result of Lachaud, Ritzenthaler and Zykin \cite{LRZ} that connects invariants of curves to Siegel modular forms. In this Section, we now generalize a result of Lockhart \cite{Lockhart} to specifically connect the discriminant of a hyperelliptic curve to the Siegel modular form $\Sigma_{140}$ of equation \eqref{eq:sigma140}. Then, we deduce the divisibility of $\Sigma_{140}$ by an odd prime $\mathfrak{p}$ to the bad reduction of the curve using a result of K{\i}l{\i}\c{c}er, Lauter, Lorenzo Garc\'ia, Newton, Ozman, and Streng \cite{KLLNOS}.

\subsection{The modular discriminant}\label{Lockhart}

We first turn our attention to the work of Lockhart,~\cite[Definition 3.1]{Lockhart}, in which the author gives a relationship between the discriminant $\Delta$ of a hyperelliptic curve of genus $g$ given by $y^2=F(x,1)$, which is related to the discriminant $D$ of the binary form $F(x,z)$ by the relation
\begin{equation}\label{Disc}
\Delta = 2^{4g} D
\end{equation}
(see \cite[Definition 1.6]{Lockhart}),
and a Siegel modular form similar to  $\Sigma_{140}$. From a computational perspective, the issue with the Siegel modular form proposed by Lockhart is that its value, as written, will be nonzero only for $Z$ a period matrix in a certain $\Gamma(2)$-equivalence class. Indeed, on page 740, the author chooses the traditional symplectic basis for $H_1(C,\mathbb{Z})$ which is given by Mumford \cite[Chapter III, Section 5]{Mumford}. If one acts on the symplectic basis by a matrix in $\Gamma(2)$, the value of the form given by Lockhart will change by a nonzero constant (the appearance of the principal congruence subgroup of level $2$ is related to the use of half-integral theta characteristics to define the form), but if one acts on the symplectic basis by a general element of $\Sp(6,\mathbb{Z})$, the value of the form might become zero.

As explained in \cite{BILV}, in general to allow for the period matrix to belong to a different $\Gamma(2)$-equivalence class, one must attach to the period matrix an element of a set defined by Poor \cite{poor}, which we call an $\eta$-map. Therefore in general one must either modify Lockhart's definition to vary with a map $\eta$ admitted by the period matrix or use the form $\Sigma_{140}$, which is nonzero for any hyperelliptic period matrix. We give here the connection between these two options. We begin by describing the maps $\eta$ that can be attached to a hyperelliptic period matrix. We refer the reader to \cite{poor} or \cite{BILV} for full details.

Throughout, let $C$ be a smooth hyperelliptic curve of genus $g$ defined over~$\mathbb{C}$ equipped with a period matrix $Z$ for its Jacobian, and for which the branch points of the degree $2$ morphism $\pi \colon C \to \mathbb{P}^1$ have been labeled with the symbols $\{1, 2, \ldots, 2g+1, \infty\}$. We note that this choice of period matrix yields an Abel--Jacobi map,
\begin{equation*}
AJ \colon \Jac(C)  \to \C^g/(\mathbb{Z}^g+Z\mathbb{Z}^g).
\end{equation*}

We begin by defining a certain combinatorial group we will need.

\begin{definition}
Let $B = \{1, 2, \ldots, 2g+1, \infty\}$. For any two subsets $S_1, S_2 \subseteq B$, we define
\begin{equation*}
S_1 \circ S_2 = (S_1 \cup S_2) - (S_1 \cap S_2),
\end{equation*}
the symmetric difference of the two sets. For $S \subseteq B$ we also define $S^c = B - S$, the complement of $S$ in $B$. Then we have that the set
\begin{equation*}
\{S \subseteq B : \# S \equiv 0 \pmod{2} \} / \{S \sim S^c\}
\end{equation*}
is a commutative group under the operation $\circ$, of order $2^{2g}$, with identity $\emptyset \sim B$. 
\end{definition}

Given the labeling of the branch points of $C$, there is a group isomorphism (see \cite[Corollary 2.11]{Mumford} for details) between the $2$-torsion of the Jacobian of $C$ and the group $G_B$ in the following manner: To each set $S \subseteq B$ such that $\# S \equiv 0 \pmod{2}$, associate the divisor class of the divisor
\begin{equation}\label{eq:2torsion}
e_S = \sum_{i \in S} P_i - (\#S) P_{\infty}.
\end{equation}

Then we can assign a map which we denote $\eta$ by sending $S \subseteq B$ to the unique vector $\eta_S$ in $(1/2)\mathbb{Z}^{2g}/\Z^{2g}$ such that $AJ(e_S)=(\eta_S)_2 + Z (\eta_S)_1$. Since there are $(2g+2)!$ different ways to label the $2g+2$ branch points of a hyperelliptic curve $C$ of genus $g$, there are several ways to assign a map $\eta$ to a matrix $Z \in \mathcal{H}_g$. It suffices for our purposes to have one such map $\eta$.

Given a map $\eta$ attached to $Z$, one may further define a set $U_{\eta} \subseteq B$:
\begin{equation*}
U_{\eta} = \{i \in B - \{\infty\} : e_*(\eta(\{i\})) = -1 \} \cup \{\infty \},
\end{equation*}
where for $\xi = \left( \begin{smallmatrix} \xi_1 & \xi_2 \end{smallmatrix} \right) \in (1/2)\Z^{2g}$, we write
\begin{equation*}
e_*(\xi) = \exp(4\pi i \xi_1^T  \xi_2),
\end{equation*}
as in equation \eqref{eq:estar}.

Then following Lockhart \cite[Definition 3.1]{Lockhart}, we define
\begin{definition}
Let $Z \in \mathcal{H}_g$ be a hyperelliptic period matrix. Then we write
\begin{equation}
\phi_{\eta}(Z) = \prod_{T \in \mathcal{I}} \vartheta[\eta_{T \circ U_{\eta}}](0,Z)^4
\end{equation}
where $\mathcal{I}$ is the collection of subsets of $\{1,2,\ldots, 2g+1,\infty\}$ that have cardinality $g+1$.
\end{definition}

\begin{remark}
We note that in this work we write our hyperelliptic curves with a model of the form $y^2 = F(x,1)$, where $F$ is of degree $2g+2$. In other words we do not require one of the Weierstrass points of the curve to be at infinity. It is for this reason that we modify Lockhart's definition above, so that the analogue of his Proposition 3.2 holds for $F$ of degree $2g+2$ rather than $2g+1$.

The Siegel modular form that we define here is equal to the one given in his Definition 3.1 for the following reason: Because $T^c \circ U_{\eta} = (T \circ U_{\eta})^c$, it follows that $\eta_{T \circ U_{\eta}} \equiv \eta_{T^c \circ U_{\eta}} \pmod{\mathbb{Z}}$. Therefore $\vartheta[\eta_{T \circ U_{\eta}}](0,Z)$ differs from $\vartheta[\eta_{T^c \circ U_{\eta}}](0,Z)$ by at worse their sign. Since we are raising the theta function to the fourth power, the sign disappears, and the product above is equal to the product given by Lockhart, in which $T$ ranges only over the subsets of $\{1,2,\ldots, 2g+1\}$ of cardinality $g+1$, but each theta function is raised to the eighth power.
\end{remark}

We now recall Thomae's formula, which is proven in~\cite{Fay,Mumford} for Mumford's period matrix, obtained using his so-called traditional choice of symplectic basis for the homology group $H_1(C,\mathbb{Z})$, and in~\cite{BILV} for any period matrix. 

\begin{theorem}[Thomae's formula]
  Let $C$ be a hyperelliptic curve defined over $\mathbb{C}$ and fix $y^2 = F(x,1) = \prod_{i = 1}^{2g+2} (x-a_i)$ a model for $C$.  Let $\Omega = \left( \begin{smallmatrix} \Omega_1 & \Omega_2\end{smallmatrix}\right)$ be a Riemann matrix for the curve obtained by integrating the forms $\omega_i$ of equation \eqref{eq:basiseta} against a symplectic basis for the homology group $H_1(C,\mathbb{Z})$ and $Z =  \Omega_2^{-1} \Omega_1 \in \mathbb{H}_g$ be the period matrix associated to this symplectic basis. Finally, let $\eta$ be an $\eta$-map attached to the period matrix $Z$. For any subset $S$ of $B$ of even cardinality, we have that
\begin{equation*}
\vartheta[\eta_{S \circ U_{\eta}}](0,Z)^4 = c\prod_{\substack{i<j \\ i,j \in S}} (a_i - a_j) \prod_{\substack{i<j \\ i,j \not \in S}} (a_i - a_j), 
\end{equation*}
where $c$ is a constant depending on $Z$ and on the model for $C$.
\end{theorem}

We now restrict our attention to the case of genus $g = 3$ which is of interest to us in this work. We note that since $Z$ is a hyperelliptic period matrix, by \cite{igusa67} a single one of its even theta constants vanishes, and therefore we have
\begin{equation*}
\phi_{\eta}(Z)  = \Sigma_{140}(Z).
\end{equation*}

We then have the following Theorem, which is a generalization to our setting of Proposition 3.2 of \cite{Lockhart} for genus 3 hyperelliptic curves:

\begin{theorem}\label{LockhartGen}
Let $C$ be a hyperelliptic curve defined over $\mathbb{C}$ and fix $y^2 = F(x,1) = \prod_{i = 1}^{2g+2} (x-a_i)$ a model for $C$.  Let $\Omega = \left( \begin{smallmatrix} \Omega_1 & \Omega_2\end{smallmatrix}\right)$ be a Riemann matrix for the curve obtained by integrating the forms $\omega_i$ of equation \eqref{eq:basiseta} against a symplectic basis for the homology group $H_1(C,\mathbb{Z})$ and $Z =  \Omega_2^{-1} \Omega_1 \in \mathbb{H}_g$ be the period matrix associated to this symplectic basis. Then
\begin{equation}
\Delta^{15} = 2^{180}\pi^{420} \det(\Omega_2)^{-140}\Sigma_{140}(Z),
\end{equation}
where we recall that $\Delta$ is the discriminant of the model that we have fixed for $C$.
\end{theorem}
 
\begin{proof}
We show how to modify Lockhart's proof. We first remind the reader that $\Delta = 2^{12}D$ by \cite[Definition 1.6]{Lockhart}, where again $D$ is the discriminant of the binary form $F(x,z)$.
Then as Lockhart does, we use Thomae's formula:
\begin{equation*}
\vartheta[\eta_{T \circ U_{\eta}}](0,Z)^4 =c \prod_{\substack{i<j \\ i,j \in T}} (a_i - a_j) \prod_{\substack{i<j \\ i,j \not \in T}} (a_i - a_j), 
\end{equation*}
if $T$ is a subset of $\{1,2,\ldots, 7,\infty\}$ of cardinality $4$. Taking the product over all such $T$, we get
\begin{equation*}
\phi_{\eta}(Z) = c^{70} \prod_{T} \left(\prod_{\substack{i<j \\ i,j \in T}} (a_i - a_j) \prod_{\substack{i<j \\ i,j \not \in T}} (a_i - a_j) \right),
\end{equation*}
since $\binom{8}{4}=70$.

We now count how many times each factor of $(a_i-a_j)$ appears on the left-hand side:
\begin{align*}
\# \{ T : i, j \in T \text{ or } i, j \not \in T\} & = \#\{T: i,j \in T\} + \# \{T: i, j \not \in T\} \\
& = \binom{6}{2} + \binom{6}{4} = 2\binom{6}{4} = 30.
\end{align*}
Therefore,
\begin{align*}
\phi_{\eta}(Z) &= c^{70} \prod_{\substack{i<j \\ i,j \in B}} (a_i - a_j)^{30},\\
& = c^{70} D^{15},\\
& = 2^{-180}c^{70} \Delta^{15}.
\end{align*}

Since $\Sigma_{140}(Z)= \phi_{\eta}(Z)$ when $Z$ is hyperelliptic, we therefore get that 
\begin{eqnarray}\label{EqWithC}
\Delta^{15} = 2^{180}c^{-70}\Sigma_{140}(Z).
\end{eqnarray}

We now compute the value of the constant $c$. We denote by $\tilde{Z}$ the period matrix associated to Mumford's so-called traditional choice of a symplectic basis for the homology group $H_1(C,\mathbb{Z})$. Lockhart showed that:
\begin{eqnarray}\label{EqLockhart}
\Delta^{15} = 2^{180}\pi^{420} \det(\tilde{\Omega}_2)^{-140}\Sigma_{140}(\tilde{Z}),
\end{eqnarray} 
where $\tilde{\Omega}_2$ is the right half of the Riemann matrix obtained by integrating the basis of forms $\omega_i$ of equation \eqref{eq:basiseta} against Mumford's traditional basis for homology.

Now consider again our arbitrary period matrix $Z$ and let $M = \left( \begin{smallmatrix} A & B \\ C & D \end{smallmatrix} \right) \in \Sp(6,\mathbb{Z})$ be such that $M \cdot \tilde{Z}=Z$. Since $\Sigma_{140}$ is a Siegel modular form of weight $140$ for $\Sp(6,\Z)$, it follows that $\Sigma_{140}(Z)=\det(C\tilde{Z}+D)^{140}\Sigma_{140}(\tilde{Z})$ and combining equations \eqref{EqWithC} and \eqref{EqLockhart}, we obtain
\begin{align*} 
c &=\pi^{-6} \det(C\tilde{Z}+D)^{2}\det(\tilde{\Omega}_2)^{2}\\
&= \pi^{-6} \det(\tilde{Z}C^T+D^T)^{2}\det(\tilde{\Omega}_2)^{2} \\
& =\pi^{-6}\det(\Omega_2)^{2}.
\end{align*}
To obtain the last equality we used the fact that $\Omega_2=\tilde{\Omega}_1C^T+\tilde{\Omega}_2D^T$.
This concludes the proof.
\end{proof}

Up to the factors of $2$ appearing in the formula, this Theorem therefore realizes Corollary \ref{cor:3.3.2}, as it connects explicitly an invariant of a hyperelliptic curve to a Siegel modular form. We furthermore note that the proof above suggests that the constant $c$ in Thomae's formula for a general period matrix $Z$ of the Jacobian of a hyperelliptic curve of genus $g$ is $\pi^{-2g} \det(\Omega_2)^{2}$. Finally, the proof of Theorem~\ref{LockhartGen} could be easily generalized to genus $g > 3$ if $\phi_{\eta}$ were shown to be a modular form for $\Sp(2g,\mathbb{Z})$. We believe this is true, but leave it as future work.

\begin{remark}
Let $n_g = \binom{2g+1}{g}$. In \cite[p. 291]{Salvati}, for $Z$ a hyperelliptic period matrix, the author defines the set $M=(\xi_1, \ldots, \xi_{n_g})$ to be the unique, up to permutations, sequence of mutually distinct even characteristics satisfying: 
\begin{equation*}
P(M)(Z)=\vartheta[\xi_1](0,Z)\vartheta[\xi_2](0,Z)\ldots \vartheta[\xi_{n_g}](0,Z)\neq 0.
\end{equation*} 
We note that the eighth power of this form is exactly the form which we denote here by $\phi_{\eta}$, since $\theta[\eta_{T\circ U_{\eta}}](0,Z) \neq 0$ if and only if $T$ has cardinality $g+1$, by Mumford and Poor's Vanishing Criterion for hyperelliptic curves \cite[Proposition 1.4.17]{poor}, and in the product giving $\phi_{\eta}$, each characteristic appears twice.

The author then introduces the modular form
\begin{equation*}
F(Z)=\sum _{\sigma \in \Sp(2g,\F_2)}P(\sigma \circ M)^8(Z),
\end{equation*}
where $\Sp(2g,\F_2)$ is the group of $2g \times 2g$ symplectic matrices with entries in $\F_2$, and where the action of $\Sp(2g,\F_2)$ on the set $M$ is given in the following manner: For $\sigma = \left( \begin{smallmatrix} A & B \\ C & D \end{smallmatrix} \right) \in \Sp(2g,\F_2)$ and $m \in (1/2)\Z^{2g} \pmod{\Z^{2g}}$ a characteristic, we write
\begin{equation*}
\sigma \circ m = \begin{pmatrix} D & - C \\ -B & A \end{pmatrix} m + \begin{pmatrix} \diag(C^TD) \\ \diag(A^TB) \end{pmatrix}.
\end{equation*}
Then $P(\sigma \circ M)$ is simply the form $P(M)$ but with each characteristic $\xi_i$ replaced with $\sigma \circ \xi_i$.

For $Z$ a hyperelliptic matrix, $P(\sigma \circ M)(Z)$ is nonzero exactly when $\sigma \circ M$ is a permutation of $M$, by definition of the set $M$. Therefore, up to a constant, $F$ is simply $\Sigma_{140}$ on the hyperelliptic locus, and therefore $F-\Sigma_{140}$ is of the form $a\chi_{18}$ for $a\in \C$.

The author then proves that $\rho(F)=D^{gn_g/(2g+1)}$, where as before $D$ is the discriminant of the binary form $F(x,z)$ such that the hyperelliptic curve is given by the equation $y^2 = F(x,1)$. We note that the power given here corrects an error in the manuscript \cite{Salvati}, and agrees with the result we obtain in this paper.
\end{remark}

\subsection{Proof of Theorem \ref{prop:newinvariant}}

We are now in a position to prove Theorem \ref{prop:newinvariant}. For simplicity, we replace $f^{\frac{140}{\gcd(k,140)}}$ with $\tilde{h}$, a Siegel modular form of weight $\tilde{k} = \frac{140k}{\gcd(k,140)}$, and let $\ell = \frac{k}{\gcd(k,140)}$. Note that $\tilde{k} = 140 \ell$ and is divisible by $4$.

Using the notation of Section \ref{Sec:ModularInvs}, the analytic Siegel modular form $\tilde{h}$ corresponds to a geometric Siegel modular form $h$ by Proposition \ref{prop:2.2.1}. Let $\Phi = \sigma(\theta^* h)$ be the corresponding invariant of the hyperelliptic curve. Then by Corollary \ref{cor:3.3.2}, if the hyperelliptic curve $y^2 = F(x,1)$ has period matrix $Z$, we have
\begin{equation*}
\Phi(F) = (2 \pi i)^{3\tilde{k}} \det(\Omega_2)^{-\tilde{k}} h(Z).
\end{equation*}
Therefore we have

\begin{align*}
j(Z) = \frac{h}{\Sigma_{140}^\ell}(Z) & = \frac{(2 \pi i)^{-3\tilde{k}} \det(\Omega_2)^{\tilde{k}} \Phi(F)}{\pi^{-420\ell} \det(\Omega_2)^{140\ell}\Disc(F)^{15\ell}}\\
& = 2^{-\frac{140k}{\gcd(k,140)}}\frac{\Phi(F)}{\Disc(F)^{15\ell}}.
\end{align*}

We note that since $\Phi$ is assumed to be an integral invariant, it does not have a denominator when evaluated at $F \in \mathbb{Z}[x,z]$. We have thus obtained an invariant as in~\cite[Theorem 7.1]{KLLNOS} (we note that \emph{loc. cit.} assumes throughout that invariants of hyperelliptic curves are integral, see the discussion between Proposition 1.4 and Theorem 1.5), having negative valuation at the prime $\mathfrak{p}$. We conclude that $C$ has bad reduction at this prime.

\section{Computing modular invariants}\label{Sec-NewInv}

In this Section, we consider certain modular functions having $\Sigma_{140}$ in the denominator. We then present a list of hyperelliptic curves of genus 3 for which we computed the primes of bad reduction. As illustration of Theorem~\ref{prop:newinvariant}, we  implemented and computed with high precision the modular functions involving the form $\Sigma_{140}$ at period matrices corresponding to curves in our list. 

\subsection{Computation of the modular invariants}

For a given hyperelliptic curve model, we used the Molin-Neurohr Magma code~\cite{Molin} to compute a first period matrix and then applied the reduction algorithm given in~\cite{KLLRSS} and implemented by Sijsling in Magma to obtain a so-called reduced period matrix that is $\Sp(2g,\Z)$-equivalent to the first matrix, but that provides faster convergence of the theta constants.

Once we obtained a reduced period matrix $Z$, using Labrande's Magma implementation for fast theta function evaluation~\cite{labrande}, we computed the $36$ even theta constants for these reduced period matrices, up to 30,000 bits of precision.
\footnote{Apart from curves (1) and (6), we could recognize these values as algebraic numbers with 15,000 bits of precision; for curve (1), we needed 30,000 bits of precision. In fact, for CM field (6), the theta constants obtained using the Magma implementation~\cite{labrande} for high precision (i.e. $\geq 30,000$ bits) were not conclusive. We therefore ran an improved implementation of the naive method to get these values up to 30,000 bits of precision, and recognized the invariants as algebraic numbers after multiplying by the expected denominators.}
 Finally, from these theta constants we computed the modular invariants that we define below. 

 To define our modular invariants, we consider the following Siegel modular forms. Let $h_4$ be the Eisenstein series of weight $4$ given by 
\begin{equation}\label{h4}
h_4(Z)=\frac{1}{2^3}\sum_{\xi}\theta[\xi]^8(Z),
\end{equation}
where $\xi$ ranges over all even theta characteristics. We denote by $\alpha_{12}$ the modular form of weight 12 defined by Tsuyumine~\cite{Tsuyumine2}:
\begin{equation}\label{alpha12}
\alpha_{12}(Z)=\frac{1}{2^3\cdot 3^2}\sum_{(\xi_i)}(\theta[\xi_1](Z)\theta[\xi_2](Z)\theta[\xi_3](Z)\theta[\xi_4](Z)\theta[\xi_5](Z)\theta[\xi_6](Z))^4,
\end{equation}
where $(\xi_1,\xi_2,\xi_3,\xi_4,\xi_5,\xi_6)$ is a maximal azygetic system of even theta characteristics. By this we mean that $(\xi_i)$ is a sextuple of even theta characteristics such that the sum of any three among these six is odd. Notice that $\alpha_{12}$ is one of the 35 generators given by Tsuyumine~\cite{Tsuyumine1} of the graded ring $A(\Gamma_3)$ of modular forms of degree 3 and cannot be written as a polynomial in Eisenstein series.

In the computations below, we consider thus the following three modular functions:

\begin{equation}\label{classInvariants}
j_1(Z) = \frac{h_4^{35}}{\Sigma_{140}}(Z), \hspace{1cm} 
j_2(Z) = \frac{\alpha_{12}^{35}}{\Sigma_{140}^3}(Z),\hspace{1cm}
j_3(Z) = \frac{h_4^{5}\alpha_{12}^{10}}{\Sigma_{140}}(Z).\hspace{1cm}
\end{equation}

\subsection{Invariants of hyperelliptic curves of genus 3}

We say that a genus 3 curve $C$ over a field $M$ has complex multiplication (CM) by an order $\O$ in a sextic CM field $K$
if there is an embedding $\O \into \End(\Jac(C)_{\overline{M}})$. 
The curves numbered (1)--(8) below are the conjectural complete list of
hyperelliptic CM curves of genus $3$ that are defined over $\Q$. 
As we mentioned in the introduction, they are taken from a list 
that can be found in \cite{KS2017}. 
We note more specifically that the curves (5), (6) and (8) were 
found by Balakrishnan, Ionica, \Kilicer, Lauter, Somoza, Streng, 
and Vincent, and (1), (2), (3), and (7) were computed by Weng~
\cite{Weng}. Moreover, the hyperelliptic model of the curve with 
complex multiplication by the ring of integers in CM field (3) 
was proved to be correct by Tautz, Top, and Verberkmoes 
\cite[Proposition 4]{TTV91}, and the hyperelliptic model of the 
curve with complex multiplication by the ring of integers in CM 
field (4) was given by Shimura and Taniyama~\cite{Shimura} (see  
Example (II) on page 76). For these examples, $\O_K$ denotes the 
ring of integers of the CM field $K$.  

The curves numbered (9)--(10) are non-CM hyperelliptic curves presented in \cite{Chabauty}. They were already used there for experiments, this time related to the Chabauty-Coleman method.  The curves numbered (11)--(13) are non-CM modular hyperelliptic curves; a list, which contains $X_0(33), X_0(39)$ and $X_0(41)$, of modular hyperelliptic curves are given by Ogg \cite{Ogg}, then Galbraith in his Ph.D thesis writes equations for these curves \cite{GalbraithThesis}. 

When we say that a prime is of bad reduction, we will mean that it is a prime of geometrically bad reduction of the curve.
For each curve below, the odd primes of bad reduction are computed using the results in \cite[Section 3]{HypRed} if $p>7$ and in Proposition 4.5 and Corollary 4.6 in~\cite{BW} if $p=3,5,7$. We denote the discriminant of a curve $C$ by $\Delta$, as before.
  
  \begin{itemize}
\item[(1)](\cite[\S 6 - 3rd ex.]{Weng}) 
Let $K = \Q[x]/(x^6+13x^4+50x^2+49)$, which is of class number 1 and contains $\Q(i)$. 
A model for the hyperelliptic curve with CM by $\O_K$ is 
\begin{equation*}
C: y^2 = x^7+1786x^5+44441x^3+278179x 
\end{equation*}
with $\Delta = -2^{18}\cdot 7^{24}\cdot 11^{12}\cdot 19^7$. The odd primes of bad reduction of $C$ are $7$ and $11$.

\mbox{}

\item[(2)] (\cite[\S 6 - 2nd ex.]{Weng}) Let $K = \Q[x]/(x^6 + 6x^4 + 9x^2 + 1)$, which is of class number $1$ and contains $\Q(i)$.  A model for the hyperelliptic curve with CM by $\O_K$ is 
\begin{equation*}
C: y^2 = x^7+6x^5+9x^3+x 
\end{equation*}
with $\Delta = - 2^{18} \cdot 3^8$. The only odd prime of bad reduction of $C$ is $3$.
\mbox{}

\item[(3)](\cite[\S 6 - 1st ex.]{Weng})  Let $K = \Q[x]/(x^6 + 5x^4 + 6x^2 + 1) = \Q(\zeta_7 + \zeta_7^{-1}, i)$, which is of class number $1$. A model for the hyperelliptic curve with CM by $\O_K$ is 
\begin{equation*}
C: y^2 = x^7+7x^5+14x^3+7x 
\end{equation*}
with $\Delta = - 2^{18} \cdot 7^7$. The curve $C$ has good reduction at each odd $p\neq 7$ and potentially good reduction at $p=7$.

\mbox{}

\item[(4)] Let $K = \Q[x]/(x^6 + 7x^4 + 14x^2 + 7) = \Q(\zeta_7)$, which is of class number 1 and contains $\Q(\sqrt{-7})$. A model for the hyperelliptic curve with CM by $\O_K$ is 
\begin{equation*}
C: y^2 = x^7-1 
\end{equation*}
with $\Delta = - 2^{12}\cdot 7^7$. The curve $C$ has good reduction at each odd $p\neq 7$ and potentially good reduction at $p=7$.

\mbox{}

\item[(5)] Let $K = \Q[x]/(x^6 + 42x^4 + 441x^2 + 847)$, which is of class number $12$ and contains $\Q(\sqrt{-7})$. 
 A model for the hyperelliptic curve with CM by $\O_K$ is 
\begin{equation*}
	C: y^2 + x^4y= - 7x^6 + 63x^4 - 140x^2 + 393x - 28
\end{equation*}
 with $\Delta = - 3^8 \cdot 5^{24} \cdot 7^{7}$. The odd primes of bad reduction of $C$ are $3$ and $5$.
\mbox{}

\item[(6)] Let $K = \Q[x]/(x^6 + 29x^4 + 180x^2 + 64)$, which is of class number $4$ and contains $\Q(i)$. A model for the hyperelliptic curve with CM by $\O_K$ is 
\begin{equation*}
C: y^2 = 1024x^7 - 12857x^5 + 731x^3 + 688x
\end{equation*}
with $\Delta = - 2^{60} \cdot 11^{24} \cdot 43^7$. The only odd prime of bad reduction of $C$ is $11$.
\mbox{}

\item[(7)] (\cite[\S 6 - 4th ex.]{Weng}) Let $K = \Q[x]/(x^6 + 21x^4 + 116x^2 + 64)$, which is of class number $4$ and contains $\Q(i)$. A model for the hyperelliptic curve with CM by $\O_K$ is 
\begin{equation*}
C: y^2 = 64x^7 - 124x^5 + 31x^3 + 31x
\end{equation*}
with $\Delta = - 2^{44} \cdot 31^{7}$. The curve has potentially good reduction at 31.

\mbox{}

\item[(8)] Let $K = \Q[x]/(x^6 +42x^4 +441x^2 +784)$, which is of class number $4$ and contains $\Q(i)$. A model for the hyperelliptic curve with CM by $\O_K$ is 
\begin{equation*}
C: y^2 = 16x^7 + 357x^5 - 819x^3 + 448x 
\end{equation*}
with $\Delta = - 2^{48} \cdot 3^8 \cdot 7^7$. 
The only odd prime of bad reduction of $C$ is $3$.

\mbox{}

\item[(9)]~(\cite{Chabauty}) The hyperelliptic curve  
\begin{equation*}
C : y^2 = 4x^7 + 9x^6 - 8x^5 - 36x^4 - 16x^3 + 32x^2 + 32x + 8
\end{equation*}
is a non-CM curve with $\Delta =2^{37} \cdot 1063$.
The only odd prime of bad reduction of $C$ is $1063$.

\mbox{}

\item[(10)]~(\cite{Chabauty}) The hyperelliptic curve  
\begin{equation*}
C : y^2 = -4x^7 + 24x^6 - 56x^5 + 72x^4 - 56x^3 + 28x^2 - 8x + 1
\end{equation*}
is a non-CM curve with $\Delta =- 2^{28} \cdot 3^4 \cdot 599$.
The odd primes of bad reduction of $C$ are $3$ and $599$.
 \mbox{}

\item[(11)]~(\cite{GalbraithThesis}\cite{Ogg}) The hyperelliptic curve  
\begin{equation*}
C : y^2 =  x^8 + 10x^6 - 8x^5 + 47x^4 - 40x^3 + 82x^2 - 44x + 33
\end{equation*}
is the modular curve $X_0(33)$.
It has $\Delta =2^{28} \cdot 3^{12} \cdot 11^6$.
The odd primes of bad reduction of $C$ are $3$ and $11$.
\mbox{}

\item[(12)] (\cite{GalbraithThesis}\cite{Ogg}) The hyperelliptic curve  
\begin{equation*}
C : y^2 =  x^8 - 6x^7 + 3x^6 + 12x^5 - 23x^4 + 12x^3 + 3x^2 - 6x + 1
\end{equation*}
is the modular curve $X_0(39)$.
It has $\Delta =2^{28} \cdot 3^8 \cdot 13^4$. The  odd primes of bad reduction of $C$ are $3$ and $13$.
\mbox{}

\item[(13)] (\cite{GalbraithThesis}\cite{Ogg}) The hyperelliptic curve  
\begin{equation*}
C : y^2 = x^8 - 4x^7 - 8x^6 + 10x^5 + 20x^4 + 8x^3 - 15x^2 - 20x - 8
\end{equation*}
 is the modular curve $X_0(41)$.
It has $\Delta = - 2^{28} \cdot 41^6$. The only odd prime of bad reduction of $C$ is $41$. 
\end{itemize}  

\vspace{0.5 cm}  
  
We recall that the discriminant $\Delta$ of a hyperelliptic curve $C$ of genus $3$ is an invariant of degree $14$ (Section $1.5$ of \cite{LerRit}). For our computations, we considered the following absolute\footnote{An \emph{absolute} invariant is a ratio of homogeneous invariants of the same degree.} invariants, derived using the Shioda invariants:
\begin{align}\label{shioda}
J_2^7/\Delta, J_3^{14}/\Delta^3, J_4^7/\Delta^2, J_5^{14}/\Delta^5, J_6^7/\Delta^3, J_7^2/\Delta, J_8^7/\Delta^4, J_9^{14}/\Delta^9, J_{10}^7/\Delta^5.
\end{align}

The numerical data in Table~\ref{tableh} shows the tight connection between the odd primes appearing in the denominators of these invariants, the odd primes of bad reduction for the hyperelliptic curve, and the odd primes dividing the denominators of $j_1,j_2$ and $j_3$. In the denominators of $j_1,j_2$ and $j_3$, we intentionally omitted the denominators of the formulae~\eqref{h4} and~\eqref{alpha12}, i.e. $2^3$ and $2^3\cdot 3^2$. Note that we do not have a proof for the fact that $h_4$ and $\alpha_{12}$ fulfill the condition in Theorem~\ref{prop:newinvariant}, i.e. that their corresponding curve invariants are integral. One can see that for all the curves we considered, a prime $\geq 3$ appears in the denominator of these modular invariants if and only if it is a prime of bad reduction for the curve. Our results are evidence that either the condition in Theorem~\ref{prop:newinvariant} is a reasonable one, or that the result in this theorem may be extended to a larger class of modular forms. 

Note that the Shioda invariants $J_2, J_3, \ldots, J_{10}$ are not integral and their denominators factor as products of powers of 2,3,5 and 7 (see~\cite{LerRit} for a set of formulae). This is the reason why these primes may appear in the denominators of the Shioda invariants, even when they are not primes of bad reduction. However, one can see that the primes $>7$ appearing in the denominators of the invariants in Eq.~\eqref{shioda} are exactly the primes of bad reduction, which confirms Theorem 7.1 in~\cite{KLLNOS}. In the Table,  all the entries marked by $-$ represent values equal to zero. 
  
\begin{landscape}
\begin{table}
  \caption{Denominators of invariants} \centering \label{tableh}
 \resizebox{\textwidth}{!}{
\begin{tabular}{|c|c|c|c|c|}
\hline
  &   & Odd primes of &  denominators & Odd primes in the deno-   \\
 \small Curve & Discriminant & bad reduction & of $j_1,j_2,j_3$ & minators of invariants in Eq.~\ref{shioda} \\
\hline\hline
\multirow{3}{*}{(1)} & \multirow{3}{*}{$-2^{18}\cdot 7^{24}\cdot 11^{12}\cdot 19^7$} & \multirow{3}{*}{$7,11$} & $-7^{80}\cdot 11^{40}$ &  $7^{31}\cdot11^{12}, - ,7^{76}\cdot11^{24}$ \\
 & & & $7^{240}\cdot 11^{120}$ & $-,7^{114}\cdot11^{36},-$\\
 & & & $7^{80}\cdot 11^{40}$  & $5^7\cdot7^{159}\cdot11^{41},-,5^7\cdot7^{197}\cdot11^{60}$\\
\hline
\multirow{3}{*}{(2)} & \multirow{3}{*}{$-2^{12}\cdot 3^8$} & \multirow{3}{*}{$3$} & 1 &  $3^8\cdot7^7,-,3^{23}\cdot7^{28}$\\
& & & $2^3\cdot 3^{12}$  & $-,3^{38}\cdot7^{42},-$\\
& & & $1$ & $3^{32}\cdot5^7\cdot7^{63},-,3^{47}\cdot5^7\cdot7^{77}$\\
\hline
\multirow{3}{*}{(3)} & \multirow{3}{*}{$-2^{18}\cdot 7^7$} &  \multirow{3}{*}{none} & $1$ & $1,-,7^{14}$ \\
& &  & $2^3$ & $-,7^{21},-$\\
& & & $1$ & $5^7\cdot7^{35},-,5^7\cdot7^{42}$\\
\hline
\multirow{3}{*}{(4)} & \multirow{3}{*}{$-2^{12}\cdot 7^7$} &  \multirow{3}{*}{none} & $1$ & $-,-,-$ \\
& &  & $2^3$ & $-,-,7^7$\\
& & & $1$ & $-,-,-$\\
\hline
\multirow{3}{*}{(5)} & \multirow{3}{*}{$-3^8\cdot 5^{24}\cdot 7^{7}$} &  \multirow{3}{*}{3,5} 
& $-$ & $3^8\cdot 5^{31},5^{100},3^{23}\cdot 5^{41}$ \\
& &  & $2^3 \cdot 3^{12} \cdot 5^{240}$ & $3^{12}\cdot 5^{120},3^{38}\cdot 5^{72},3^6\cdot 5^{26}$\\
& & & $-$ & $3^{32}\cdot 5^{103},3^{72}\cdot 5^{216},3^{47}\cdot 5^{120}$\\
\hline
\multirow{3}{*}{(6)} & \multirow{3}{*}{$-2^{60}\cdot11^{24}\cdot 43^{7}$} &  \multirow{3}{*}{11} 
& $2^{125}\cdot 11^{80}$ & $7^7\cdot11^{24}, -, 7^{28}\cdot11^{48}$ \\
& &  & $2^{413}\cdot 11^{240}$ & $-,  7^{42}\cdot11^{72}, -$\\
& & & $2^{135}\cdot 11^{80}$ & $ 5^7\cdot7^{77}\cdot11^{96}, -, 5^7\cdot7^{77}\cdot11^{120}$\\
\hline
\multirow{3}{*}{(7)} & \multirow{3}{*}{$-2^{44}\cdot 31^{7}$} &  \multirow{3}{*}{none} & $2^{25}$ & $7^7,-,7^{28}$ \\
& &  & $2^{113}$ & $-,7^{42}, -$\\
& & & $2^{35}$ & $5^7\cdot7^{63},-,5^7\cdot7^{77}$\\
\hline
\multirow{3}{*}{(8)} & \multirow{3}{*}{$-2^{48}\cdot 3^{8}\cdot 7^{7}$} &  \multirow{3}{*}{3} 
     & $2^{85}$ & $ 3^{8}, -,  3^{23}\cdot7^{14}$ \\
& &  & $2^{293}\cdot 3^{12}$ & $-, 3^{38}\cdot7^{21}, -$\\
 & & & $2^{95}$ & $ 3^{32}\cdot5^7\cdot7^{35}, -,  3^{47}\cdot5^7\cdot7^{42}$\\
\hline
\multirow{3}{*}{(9)} & \multirow{3}{*}{$2^{37} \cdot 1063$} &  \multirow{3}{*}{1063} 
     & $ 2^{60} \cdot 1063^{15}$ & $ 5^7 \cdot 7^7 \cdot 1063,  5^{28} \cdot7^{42} \cdot1063^3, 3^7\cdot 7^{28}\cdot 1063^2, $ \\
& &  & $1063^{10}$ & $5^{14}\cdot 7^{70}\cdot 1063^5, 3^{14}\cdot 7^{35}\cdot 1063^3,5^2\cdot 7^{14}\cdot 1063, $\\
 & & & $1063^5$ & $3^7\cdot 5^{14}\cdot 7^{63}\cdot 1063^4, 5^{14}\cdot 7^{126}\cdot 1063^9, 3^{14}\cdot 5^{14}\cdot 7^{70}\cdot 1063^5$\\
\hline
\multirow{3}{*}{(10)} & \multirow{3}{*}{$- 2^{28}\cdot 3^4 \cdot 599$} &  \multirow{3}{*}{3,599} 
     & $ 599^{15}$ & $5^7 \cdot 7^7\cdot 599, 5^{28}\cdot 7^{42}\cdot 599^3, 3\cdot 7^{28}\cdot 599^2, $ \\
& &  & $3^5\cdot 599^{10}$ & $3^6\cdot  5^{14}\cdot 7^{70}\cdot 599^5, 7^{42}\cdot 599^3,7^{14}\cdot 599, $\\
 & & & $599^5$ & $3^2\cdot 5^7\cdot 7^{63}\cdot 599^4,  5^{14}\cdot 7^{126}\cdot 599^9, 5^{14}\cdot 7^{77}\cdot 599^5$\\
\hline
\multirow{3}{*}{(11)} & \multirow{3}{*}{$2^{28}\cdot 3^{12}\cdot 11^6$} &  \multirow{3}{*}{3,11} 
     & $ 3^{40}\cdot 11^{90}$ & $3^{12}\cdot 5^7\cdot 11^6, \cdot3^8 \cdot5^{28}\cdot 7^{42}\cdot 11^{18},3^{31}\cdot 7^{21}\cdot 11^{12},$ \\
& &  & $3^{50}\cdot 11^{60}$ & $3^{60} \cdot 5^{14}\cdot 7^{56}\cdot 11^{30}, 3^{50}\cdot 7^{42}\cdot 11^{18}, 3^{14}\cdot 7^{12}\cdot 11^6,$\\
 & & & $3^{20}\cdot 11^{30}$ & $3^{41}\cdot 5^7\cdot 7^{49}\cdot 11^{24},3^{136}\cdot 5^{14}\cdot 7^{126}\cdot 11^{54}, 3^{60} \cdot 5^{14}\cdot 7^{70}\cdot 11^{30}$\\
\hline
\multirow{3}{*}{(12)} & \multirow{3}{*}{$2^{28}\cdot3^8\cdot13^4$} & \multirow{3}{*}{3,13} &  $2^{135} \cdot 3^{120} \cdot 13^{60}$& $3\cdot5^7\cdot 7^7\cdot 13^4, 3^{10}\cdot 5^{28}\cdot 7^{42}\cdot 13^{12}, 7^{28}\cdot 13^8$\\
& & &  $3^{45}\cdot13^{40}$ & $5^{14} \cdot7^{70}\cdot 13^{20},7^{42}\cdot 13^{12}, 5^2\cdot 7^{14}\cdot 13^2$\\
& & & $3^{30} \cdot 13^{20}$ & $5^{14}\cdot 7^{63}\cdot 13^{16}, 5^{14}\cdot 7^{126}\cdot 13^{36}, 5^{14} \cdot7^{77}\cdot 13^{20}$\\
\hline 
\multirow{3}{*}{(13)} & \multirow{3}{*}{$ - 2^{28}\cdot  41^6$} & \multirow{3}{*}{41} &  $ 2^{135}\cdot  41^{90}$  & $7^7 \cdot41^6, 7^{42}\cdot 41^{18},3^7 \cdot7^{28}\cdot 41^{12}$\\
& & & $41^{60}$ & $7^{70}\cdot 41^{30}, 3^{14}\cdot 7^{42}\cdot 41^{18}, 3^2 \cdot7^{14}\cdot 41^4$\\
& & & $41^{30}$ & $3^7 \cdot5^7\cdot 7^{63} \cdot41^{24},3^{28} \cdot7^{126}\cdot 41^{54}, 3^{14}\cdot 5^7 \cdot7^{77}\cdot 41^{30}$ \\
\hline
\end{tabular}}
\end{table} 
\end{landscape}

We note that because of its large weight, $\Sigma_{140}$ is expensive to compute, so the modular invariants computed here may not be the most convenient to use from a computational point of view. As suggested by Lockhart \cite[p. 741]{Lockhart}, it might be worth finding a Siegel modular form that corresponds to a lower power of the discriminant, especially if one is to pursue further the goal of finding modular expressions for the Shioda invariants.
We note that Tsuyumine~\cite{Tsuyumine1} introduced the modular form $\chi_{28}$ of weight 28 such that $\rho(\chi_{28})=D^3$, where as earlier $D$ is the discriminant of the binary form $F(x,z)$ such that the hyperelliptic curve is given by $y^2 = F(x,z)$.
The reason for which we chose to work with $\Sigma_{140}$ in the computations is because it was straightforward to implement.
 
Finally, we note that in the non-hyperelliptic curve case, one could show with similar reasoning as in Theorem~\ref{prop:newinvariant} that a modular function having a power of $\chi_{18}$ in the denominator, when evaluated at a plane quartic period matrix, has denominator divisible by the primes of bad reduction or of hyperelliptic reduction of the curve associated to the period matrix. In this direction, a relationship between $\chi_{18}$ and the discriminant of the non-hyperelliptic curve was shown by Lachaud, Ritzenthaler, and Zykin~\cite[Theorem 4.1.2, Klein's formula]{LRZ}.

\section{Conclusion}
We have displayed a connection between the values of certain geometric modular forms of even weight restricted to the hyperelliptic locus and the primes of bad reduction of hyperelliptic curves. A complete description of the Shioda invariants of hyperelliptic curves in terms of modular forms deserves further investigation. However, our result, combined with the bounds obtained in~\cite{KLLNOS} on primes of bad reduction for hyperelliptic curves, yields a bound on the primes appearing in the denominators of modular invariants.

\bibliographystyle{plain}
\bibliography{sample}

\begin{thebibliography}{10}

\bibitem{Chabauty}
Jennifer~S. Balakrishnan, Francesca Bianchi, Victoria Cantoral~Farf\'an, Mirela
  \c{C}iperiani, and Anastassia Etropolski.
\newblock Chabauty-{C}oleman experiments for genus 3 hyperelliptic curves.
\newblock {\em Preprint}, 2018.

\bibitem{BILV}
Jennifer~S. Balakrishnan, Sorina Ionica, Kristin Lauter, and Christelle
  Vincent.
\newblock Constructing genus-3 hyperelliptic {J}acobians with {CM}.
\newblock {\em LMS J. Comput. Math.}, 19(suppl. A):283--300, 2016.

\bibitem{basson}
Romain Basson.
\newblock {\em Arithm{\'e}tique des espaces de modules des courbes
  hyperelliptiques de genre $3$ en caract\'eristique positive}.
\newblock PhD thesis, Universit\'e de Rennes 1, Rennes, 2015.

\bibitem{BW}
Irene~I. Bouw and Stefan Wewers.
\newblock Computing {$L$}-functions and semistable reduction of superelliptic
  curves.
\newblock {\em Glasg. Math. J.}, 59(1):77--108, 2017.

\bibitem{Fay}
John~D. Fay.
\newblock {\em Theta {functions on Riemann surfaces}}, volume 352 of {\em
  Lecture Notes in Mathematics}.
\newblock Springer-Verlag, 1973.

\bibitem{GalbraithThesis}
Steven~D. Galbraith.
\newblock {\em Equations {F}or {M}odular {C}urves}.
\newblock PhD thesis, University of Oxford, 1996.

\bibitem{GorenLauter07}
Eyal~Z. Goren and Kristin~E. Lauter.
\newblock Class invariants for quartic {CM} fields.
\newblock {\em Ann. Inst. Fourier (Grenoble)}, 57(2):457--480, 2007.

\bibitem{ichikawa1}
Takashi Ichikawa.
\newblock On {T}eichm{\"u}ller modular forms.
\newblock {\em Mathematische Annalen}, 299(1):731--740, 1994.

\bibitem{ichikawa2}
Takashi Ichikawa.
\newblock Teichm{\"u}ller modular forms of degree 3.
\newblock {\em American Journal of Mathematics}, pages 1057--1061, 1995.

\bibitem{ichikawa3}
Takashi Ichikawa.
\newblock Theta constants and {T}eichm{\"u}ller modular forms.
\newblock {\em journal of number theory}, 61(2):409--419, 1996.

\bibitem{ichikawa4}
Takashi Ichikawa.
\newblock Generalized {T}ate curve and integral {T}eichm{\"u}ller modular
  forms.
\newblock {\em American Journal of Mathematics}, 122(6):1139--1174, 2000.

\bibitem{igusa67}
Jun-ichi Igusa.
\newblock Modular forms and projective invariants.
\newblock {\em Amer. J. Math.}, 89:817--855, 1967.

\bibitem{KLLRSS}
P{\i}nar K{\i}l{\i}\c{c}er, Hugo Labrande, Reynald Lercier, Christophe
  Ritzenthaler, Jeroen Sijsling, and Marco Streng.
\newblock {P}lane quartics over $\mathbb{Q}$ with complex multiplication.
\newblock {\em Preprint}, 2017.

\bibitem{KLLNOS}
P{\i}nar K{\i}l{\i}\c{c}er, Kristin~E. Lauter, Elisa Lorenzo~Garc\'ia, Rachel
  Newton, Ekin Ozman, and Marco Streng.
\newblock A bound on the primes of bad reduction for {CM} curves of genus 3.
\newblock https://arxiv.org/abs/1609.05826, 2016.

\bibitem{KS2017}
P{\i}nar K{\i}l{\i}\c{c}er and Marco Streng.
\newblock Rational {CM} points and class polynomials for genus three.
\newblock In preparation, 2016.

\bibitem{labrande}
Hugo Labrande.
\newblock Code for fast theta function evaluation.
\newblock \url{https://hlabrande.fr/math/research/}, 2017.
\newblock [accessed 2017-11-17].

\bibitem{LRZ}
Gilles Lachaud, Christophe Ritzenthaler, and Alexey Zykin.
\newblock Jacobians among abelian threefolds: a formula of {K}lein and a
  question of {S}erre.
\newblock {\em Math. Res. Lett.}, 17(2):323--333, 2010.

\bibitem{HypRed}
Reynald Lercier, Qing Liu, Elisa Lorenzo~Garc\'ia, and Christophe Ritzenthaler.
\newblock Reduction type of non-hyperelliptic genus 3 curves.
\newblock {\em Preprint}, 2018.

\bibitem{LerRit}
Reynald Lercier and Christophe Ritzenthaler.
\newblock Hyperelliptic curves and their invariants: geometric, arithmetic and
  algorithmic aspects.
\newblock {\em Journal of Algebra}, 372:595--636, 2012.

\bibitem{Lockhart}
Paul Lockhart.
\newblock On the discriminant of a hyperelliptic curve.
\newblock {\em Trans. Amer. Math. Soc.}, 342(2):729--752, 1994.

\bibitem{Molin}
Pascal Molin and Christian Neurohr.
\newblock Hcperiods, 2018.

\bibitem{Mumford1}
David Mumford.
\newblock {\em Tata lectures on theta. {I}}.
\newblock Modern Birkh\"auser Classics. Birkh\"auser Boston, Inc., Boston, MA,
  2007.
\newblock With the collaboration of C. Musili, M. Nori, E. Previato and M.
  Stillman, Reprint of the 1983 edition.

\bibitem{Mumford}
David Mumford.
\newblock {\em Tata lectures on theta. {II}}.
\newblock Modern Birkh\"auser Classics. Birkh\"auser Boston, Inc., Boston, MA,
  2007.
\newblock Jacobian theta functions and differential equations, With the
  collaboration of C. Musili, M. Nori, E. Previato, M. Stillman, and H.
  Umemura, Reprint of the 1984 original.

\bibitem{Ogg}
Andrew~P. Ogg.
\newblock Hyperelliptic modular curves.
\newblock {\em Bull. Soc. Math. France}, 102:449--462, 1974.

\bibitem{poor}
Cris Poor.
\newblock The hyperelliptic locus.
\newblock {\em Duke Math. J.}, 76(3):809--884, 1994.

\bibitem{Salvati}
Ricardo Salvati~Manni.
\newblock Slope of cusp forms and theta series.
\newblock {\em Journal {of Number Theory}}, 83:282--296, 2000.

\bibitem{Shimura}
Goro Shimura.
\newblock {\em Abelian varieties with complex multiplication and modular
  functions}, volume~46 of {\em Princeton Mathematical Series}.
\newblock Princeton University Press, Princeton, NJ, 1998.

\bibitem{Shioda}
Tetsuji Shioda.
\newblock On the graded ring of invariants of binary octavics.
\newblock {\em Amer. J. Math.}, 89:1022--1046, 1967.

\bibitem{TTV91}
Walter Tautz, Jaap Top, and Alain Verberkmoes.
\newblock Explicit hyperelliptic curves with real multiplication and
  permutation polynomials.
\newblock {\em Canad. J. Math.}, 43(5):1055--1064, 1991.

\bibitem{Tsuyumine1}
S.~Tsuyumine.
\newblock On the {S}iegel modular function field of degree three.
\newblock {\em Compositio Math.}, 63(1):83--98, 1987.

\bibitem{Tsuyumine2}
Shigeaki Tsuyumine.
\newblock {On Siegel modular forms of degree 3}.
\newblock {\em {Amer. J. Math.}}, 108:755--862, 1986.

\bibitem{Weng}
Annegret Weng.
\newblock A class of hyperelliptic {CM}-curves of genus three.
\newblock {\em J. Ramanujan Math. Soc.}, 16(4):339--372, 2001.

\end{thebibliography}
\end{document}